\documentclass[fleqn]{article}
\usepackage[totalwidth=7.0in,totalheight=9.25in]{geometry}

\usepackage{amsmath}
\usepackage{amssymb}
\usepackage{amsthm}
\usepackage{bm}
\usepackage{caption}
\usepackage[mathscr]{euscript}
\usepackage[pdftex]{hyperref}
\usepackage{mathrsfs}
\usepackage{mathtools}
\usepackage{titlesec}
\usepackage{xcolor}

\usepackage{enumitem}

\definecolor{refcolor}{RGB}{42,93,176}
\hypersetup{colorlinks=true,linkcolor=refcolor,citecolor=refcolor}

\titleformat*{\section}{\large\bfseries}
\titlespacing*\section{0pt}{10pt plus 2pt minus 1pt}{2pt plus 0pt minus 0pt}
\titleformat*{\subsection}{\bfseries}
\titlespacing*\subsection{0pt}{8pt plus 2pt minus 1pt}{1pt plus 0pt minus 0pt}
\newtheorem{theorem}{Theorem}[section]
\newtheorem{lemma}[theorem]{Lemma}
\newtheorem{proposition}{Proposition}
\theoremstyle{definition}
\newtheorem{definition}[theorem]{Definition}

\newcommand{\alphalist}{\renewcommand{\theenumi}{\alph{enumi}}\renewcommand*\labelenumi{(\theenumi)}}

\newcommand\bbR{{\mathbb R}}
\newcommand\onm\operatorname
\newcommand\sset[1]{\{\mskip1.25mu#1\mskip1.25mu\}}
\newcommand\set[2]{\{\mskip1mu #1\mid#2\mskip1mu\}}
\newcommand\clB{{\mathcal B}}
\newcommand\clC{{\mathcal C}}
\newcommand\clX{{\mathcal X}}
\newcommand\clY{{\mathcal Y}}
\newcommand\clN{{\mathcal N}}

\begin{document}\allowdisplaybreaks\mathtoolsset{showonlyrefs}

\thispagestyle{empty}
\setlist{nosep}

\begin{raggedright}
{\LARGE\bf
  The geometry of convergence in numerical analysis
}
\\[15pt]{
  George W.\ Patrick\\
  Department of Mathematics and Statistics,\\
  University of Saskatchewan, S7N5E6\\
  \today
}
\\[10pt]
{\color{lightgray}\rule{\textwidth}{.5pt}}
\parbox{\textwidth}{
  \vspace*{8pt}{\noindent\bf Abstract\\}
  The domains of mesh functions are strict subsets of the underlying space of
continuous independent variables. Spaces of partial maps between topological
spaces admit topologies which do not depend on any metric. Such topologies
geometrically generalize the usual numerical analysis definitions of
convergence.

}
\\[9pt]
{\color{lightgray}\rule{\textwidth}{.5pt}}\\[6pt]
\end{raggedright}

\section{Introduction}
\label{section:introduction}

The numerical analysis of ordinary differential equations is generally cast
within the context of open subsets of $\bbR^n$. This corresponds to the machine
implementations, where almost everything is anyway a tuple of floating point
numbers. However, the familiar $\bbR^n$ structures may lead to constructs that
are not deeply natural, and to unarchitected mathematics. And there is
significant numerical literature specifically targeting the differentiable
category\,---\,see the references within and citations
to~\cite{
  IserlesA-Munthe-KassHZ-NorsettSP-ZannaA-2000-1,   %
  MarsdenJE-PatrickGW-ShkollerS-1998-1,             %
  MarsdenJE-WestM-2001-1%
}.

Consider what it means for a discrete approximation~$y_{h,k}$ to converge to a
solution $y(t)$ of an ordinary differential equation ($k$ is an integer and the
sequence $y_{h,k}$ approximates $y(t)$ at time $t=kh$). Not possible is pointwise convergence, i.e.,
\begin{equation*}
  \lim_{h\to0^+}y_h(t)=y(t),\qquad t\in\onm{domain}(y),
\end{equation*}
where $y_h$ is the (partial) function defined to have values~$y_{h,k}$ at
$t_{h,k}=kh$, because the~$y_h$ are not defined at every~$t$ in the continuum
for which $y(t)$ is defined. So, in general practice
\cite{AscherUM-PetzoldLR-1998-1,
IserlesA-2009-1,
SchatzmanM-2002-1,
SuliE-MayerDF-2003-1},
convergence means something more like
\begin{equation}\label{eq:usual-convergence}
  \lim_{h\rightarrow 0^+}\max_{k=0,1,\ldots,\lfloor t^*/h\rfloor}
  \|y_{h,k}-y(t_{h,k})\|=0.
\end{equation}
Expendiently, the smooth solution has been evaluated at varying times
determined by the domains of the approximations, and one of the $\bbR^n$ norms
utilized to size the differences.

In~\eqref{eq:usual-convergence}, an apparent dependence on some particular norm
has emerged at the level of the most basic definition. The definition does not
cleanly export to the coordinate-invariant context of differentiable
manifolds. A somewhat more sophisticated observation is
that~\eqref{eq:usual-convergence} is inherently asymmetric between the limit
precursors and the target, such as one has in the definition of a Cauchy
sequence, which symmetrically refers to two precursors.

Had there been no distance available in the first place, then the following may
have been more visible: \emph{a sequence of mesh functions $y_n$ converges to a
  solution $y$ if}
\begin{equation}\label{eq:evaluation-convergence-0}
  \parbox{.85\textwidth}{\smallskip\it
    $\displaystyle\lim_{n\to\infty}y_n(t_n)=y(t)$
    whenever $t_n\in\onm{domain}(y_n)$ and $t_n\to t$.}
\end{equation}
This \emph{limit-evaluation convergence} (the evaluation of the limit of
functions is the limit the evaluations) is metric independent. It even exports
to the category of topological spaces, because it does not depend on any smooth
structure. Given a uniformity, it could support a notion of a Cauchy sequence.

But there is reason to favour topologically defined convergence, as opposed to
relying solely on some apriori convergence criteria. For example, it is
generally sufficient in applications to use uniform convergence of derivatives
on compact sets in distributional test function spaces. The natural inductive
limit carrier of this convergence (see
Theorem~\ref{th:membership-convergence-generates-inductive-limit}) is not a
locally convex topological vector space. Of interest in the theory of
distributions is the dual space, but one does not have a Hahn-Banach theorem
without local convexity (\cite{RudinW-1973-1}, esp.~Example~1.47). In fact,
addition of test functions is not continuous unless the local convexity problem
is addressed.

It is common to test a numerical method by machine verification from some fixed
initial condition to some fixed final time. This notion is of course
invariant. But used as a theoretical foundation of convergence, 
the result is not localizable, meaning that if the numerical approximations are
restricted to an arbitrary open subinterval, then that restriction inherently
has varying start values and initial conditions. Local theorems depending on
fixed initial conditions and fixed final time do not cleanly globalize because
patching inherently breaks that constraint.

\section{The convergence criteria}
\label{section:convergence-criteria}

Quite apart from its relation to any topology, the
criteria~\eqref{eq:evaluation-convergence-0} is itself somewhat
dysfunctional. Consider the sequence of partial maps
\begin{equation}\label{eq:evaluation-convergence-eg}
  g_i\colon(-1)^{i+1}[0,1]\rightarrow\bbR,\quad g_i(t)=(-1)^{i+1}t.
\end{equation}
The logical predicate ``$t_i\in(-1)^{i+1}[0,1]$ and $t_i\rightarrow t$'' is
empty unless $t=0$, so that~\eqref{eq:evaluation-convergence-0} is satisfied
vacuously for all $t\ne0$. Under~\eqref{eq:evaluation-convergence-0}, the
sequence $g_i$ converges to any function $g$ such that $g(0)=0$.  An
improvement may be obtained by inserting subsequences, and by restricting the
domain of the limit function:
\begin{equation}\label{eq:evaluation-convergence-1}
  \parbox{.85\textwidth}{\smallskip\it
    (a)~for all strictly increasing $i_j$, if
    $t_j\in\onm{domain}(y_{i_j})$ and $t_j\to t\in\onm{domain}(y)$ then
    $\lim_{j\to\infty}y_{i_j}(t_j)=y(t)$; and~(b)~for all $t\in\onm{domain}y$
    there is a sequence~$t_i\in\onm{domain} y_i$ such that $t_i\rightarrow t$.}
\end{equation}
With this, the $g_i$ in~\eqref{eq:evaluation-convergence-eg} converge to
$g(x)=0$ with domain $\sset{0}$.

About~\eqref{eq:evaluation-convergence-1} a further issue may be apparent: it
allows the possibility of a subsequence $t_j\in\onm{domain}(y_{i_j})$ such that
$t_j$ converges to~$t\not\in \onm{domain}(y)$ and~$y_{i_j}(t_j)$ converges. In
the context of a numerical method, this would correspond to the convergence of
the method where the target limit does not exist, i.e.\ to a false numerical
prediction that a solution exists. The following modification avoids this:
\smallskip
\begin{equation}\label{eq:evaluation-convergence-2}
  \parbox{.85\textwidth}{\smallskip\it
    (a)~for all strictly increasing $i_j$, if
    $t_j\to t$ and $y_{i_j}(t_j)$ both converge
    then $t\in\onm{domain}(y)$ and $\lim_{j\to\infty}y_{i_j}(t_j)=y(t)$;
    and~(b)~for all $t\in\onm{domain}y$ there is a sequence~$t_i\in\onm{domain}
    y_i$ such that $t_i\rightarrow t$.}
\end{equation}
The criteria~\eqref{eq:evaluation-convergence-2} ensures that the domain
of the limit includes times for which the precursors converge.
Rephrasing in terms of a topological spaces $\clX$ and $\clY$, one arrives at
\begin{equation}\label{eq:evaluation-convergence-3}
  \parbox{.85\textwidth}{\smallskip\it
    $f_i\colon A_i\subseteq\clX\rightarrow\clY$ converges to $f\colon
    A\subseteq\clX\rightarrow\clY$ if (a)~for all strictly increasing $i_j$ and
    all $x_j\in A_{i_j}$ such that $x_j$ and $f_{i_j}(x_j)$ both converge,
    $\lim_j x_j\in A$ and $\lim_j f_{i_j}(x_j)=f(\lim_j x_j)$; and~(b)~for all
    $x\in A$ there is a sequence~$x_i\in A_i$ such that $\lim_i x_i=x$.}
\end{equation}
Ideally, one seeks a topology on the space of partial maps, with an
explicitly known and easily visualized neighbourhood base, within which a
sequence of partial maps converges if and only
if~\eqref{eq:evaluation-convergence-3} holds.

\section{Geometric Hypertopologies}
\label{section:geometric-hypertopologies}

Spaces of partial maps are related to spaces of subsets of a set\,---\,the
power set\,---\,because partial maps may be identified with their
graphs. There is substantial literature about topologies of the power set of a
topological space, and the related spaces of partial~maps of continuous
maps~\cite{
BackK-1986-1,                               %
BeerG-1993-2,                               %
BeerG-CasertaA-DiMaioG-LucchettiR-2014-1,   %
ChabautyC-1950-1,                           %
FellJMG-1962-1,                             %
IllanesA-NadlerSB-1999-1,                   %
KuratowskiK-1955-1,                         %
KuratowskiK-1966-1,                         %
KuratowskiK-1968-1,                         %
MichaelE-1951-1,                            %
NadlerSB-1978-1%
}.
This literature may not be so well-known and accessible in the numerical
analysis community; this section provides a review of the essential aspects.

Let $\clX$ be a topological space. A \emph{hypertopology} of $\clX$ is a
topology on a subset of $2^\clX$; a \emph{hyperspace} is such a topological
space. The \emph{geometric hypertopologies} are those defined using the only
topology on $\clX$, as opposed to for example using a metric on $\clX$.

It may be useful to have in mind the following example. If by definition
$\lim_i A_i=\set{\lim x_i}{x_i\in A_i}$ (membership in the limit is the limit
of membership), then the sequence $(-1)^i[0,1]\to\sset{0}$ while the odd and
even subsequences converge to $[-1,0]$ and $[0,1]$, respectively. Simple-minded
criteria may not be so useful.

\subsection{Kuratowski-Painlev\'e convergence}
\label{section:kuratowski-painleve-convergence}

Although here they will be the objective because the interest is numerical
analysis, sequence convergence does not generally suffice for topology. Net
convergence, where the index set is generalized to a directed ordered set, does
suffice. Net convergence is a elementary topic that is available in many texts
e.g.~\cite{RundeV-2005-1,WillardS-1970-1}. For convenience, some relevant
aspects, particularly focused on utilizing net convergence to generate
topologies, are collected in an appendix to this article.

\begin{definition}\label{df:KP-convergence}\mbox{}
\begin{enumerate}\alphalist
\item
  $x\in\clX$ is a \emph{limit point} of a net of subsets
  $A_\lambda\subseteq\clX$ if, for all open $U\ni x$, $A_\lambda\cap
  U\ne\emptyset$ finally i.e.~for all open $U\ni x$ there is a 
  $\lambda^*$ such that $A_\lambda\cap U\ne\emptyset$ for all
  $\lambda\ge\lambda^*$.  The \emph{lower closed limit} or \emph{ Kuratowski
  limit inferior} of a net $A_\lambda$ is the set of its limit points,
  denoted $\onm{Li}_\lambda A_\lambda$.
\item
  $x\in\clX$ is a \emph{cluster point} of a net of subsets
  $A_\lambda\subseteq\clX$ if, for all open $U\ni x$, $A_\lambda\cap
  U\ne\emptyset$ cofinally i.e.~for all open $U\ni x$ and all $\lambda^*$ there
  is a $\lambda\ge\lambda^*$ such that
  $A_\lambda\cap U\ne\emptyset$. The \emph{upper closed limit} or \emph{
  Kuratowski limit superior} of a net $A_\lambda$ is the set of its cluster
  points, denoted $\onm{Ls}_\lambda A_\lambda$.
\item
  A net of subsets $A_\lambda$ \emph{Kuratowski-Painlev\'e converges to $A$} if
  $\onm{Li}A_\lambda=A$ and $\onm{Ls}A_\lambda=A$, denoted
  $A=\onm{K-lim}A_\lambda$.
\end{enumerate}
\end{definition}

The notations in Definition~\ref{df:KP-convergence} conform to
\cite{BeerG-1993-2}. Obviously, $\onm{Li} A_\lambda\subseteq\onm{Ls}
A_\lambda$. $\onm{Li} A_\lambda$ and $\onm{Ls} A_\lambda$ are both closed: If
$x\in\onm{cl}(\onm{Li} A_\lambda)$ and $U\ni x$ is open then $U\cap\onm{Li}
A_\lambda\ne\emptyset$. Choose $y\in U\cap\onm{Li} A_\lambda$.  Then $y\in U$,
so there is a $\lambda^*$ such that $A_\lambda\cap U\ne\emptyset$, which
suffices to show that $x\in\onm{Li} A_\lambda$. The proof for $\onm{Ls}
A_\lambda$ is similar.

Limits and cluster points are defined in terms of neighbourhoods of the
underlying topology, but they have equivalent expressions in terms of
net convergence.

\begin{proposition}\label{pp:limit-points-and-limits}
If $A_\lambda$ is a net of subsets of $\clX$ then
\begin{enumerate}\alphalist
  \item
    $x\in\clX$ is a cluster point of $A_\lambda$ if and only if there is a
    subnet $A_{\lambda_\mu}$ and a net $x_\mu\in A_{\lambda_\mu}$ such that
    $x_\mu\to x$.
  \item
    $x\in\clX$ is limit point of $A_\lambda$ if and only if it is a cluster
    point of every subnet of $A_\lambda$.
\end{enumerate}
If $\clX$ is first countable and $A_i$ is a sequence of subsets of $\clX$, then
\begin{enumerate}\alphalist\setcounter{enumi}{2}
  \item
    $x$ is a cluster point of $A_i$ if and only if there is a strictly
    increasing $i_j$ and $x_j\in A_{i_j}$ such that $x_j\to x$.
  \item
    $x$ is a limit point of $A_i$ if and only if there is a sequence $x_i\in
    A_i$ such that $x_i\to x$.
\end{enumerate}
\end{proposition}

\begin{proof}
(a)
If $x$ is a cluster point of $A_\lambda$ then the set pairs
$\set{(\lambda,U)}{A_\lambda\cap U\ne\emptyset}$ (with the ordering
$(U_1,\lambda_1)\ge (U_2,\lambda_2)$ if $U_1\subseteq U_2$ and
$\lambda_1\ge\lambda_2$) is directed. $A_{\lambda,U}$ is a subnet of
$A_\lambda$ and picking $x_{(\lambda,U)}\in A_\lambda\cap U$ provides a
suitable net converging to $x$. Conversely, suppose $A_{\lambda_\mu}$ is a
subnet of $A_\lambda$, $x_\mu\in A_{\lambda_\mu}$ and $x_\mu\to x$, and let
$\lambda^*\in\Lambda$. Given an open $U\ni x$ there is a $\mu_1^*$ such that
$x_\mu\in U$ if $\mu\ge\mu_1^*$, and there is a $\mu_2^*$ such that
$\lambda_\mu\ge\lambda^*$, for $\mu\ge\mu_2^*$. Choose $\mu$ such that
$\mu\ge\mu_1^*$ and $\mu\ge\mu_2^*$. For that 
$\mu$, $x_\mu\in A_{\lambda_\mu}\cap U$, providing $\lambda_\mu\ge\lambda^*$
and $A_{\lambda_\mu}\cap U\ne\emptyset$.

(b)
If $x$ is a limit point of $A_\lambda$ then it is a limit point of every subnet
of $A_\lambda$, and every limit point is a cluster point. Conversely, suppose
that $x$ is not a limit point. Then there is an open $U\ni x$ such that for all
$\lambda^*$ there is a $\lambda\ge\lambda^*$ such that $A_\lambda\cap
U=\emptyset$, and the set $\set{\lambda}{A_\lambda\cap U=\emptyset}$ provides a
subnet of $A_\lambda$ which has no subnet that converges to $x$.

(c)
Suppose $x$ is a cluster point of $A_i$ and let $U_j$ be a countable
neighbourhood base at $x$. Choose $i_j$ strictly increasing such that
$A_{i_j}\cap U_j\ne\emptyset$ and choose $x_{i_j}$ is that set. The converse is
immediate from~(1) because any subsequence
of $A_i$ is a subnet of that.

(d)
Suppose $x$ is a limit point of $A_i$ and let let $U_j$ be a countable
neighbourhood base at $x$.  For all~$j$ there is $N_j$ such that $i\ge N_j$
implies $A_i\cap U_j\ne\emptyset$. Without loss of generality assume $N_j$ is
increasing. Inductively choosing a sequence $x_i\in A_i$ such that $x_i\in U_1$
for $N_1\le i< N_2$, $x_i\in U_2$ for $N_2\le i<N_3$, and so on, provides an
$x_i\in A_i$ such that $x_i\to x$. Conversely, any such sequence obviously
provides final nonempty intersection of $A_i$ with any neighbourhood of
$x$.
\end{proof}

From $\onm{Li}A_\lambda\subseteq\onm{Ls}A_\lambda$ follows that
$A=\onm{K-lim}A_i$ if and only if $A\subseteq\onm{Li}A_\lambda$ and
$\onm{Ls}A_\lambda\subseteq A$ (\cite{BeerG-1993-2}, Lemma~5.2.4). In the case
that $\clX$ is first countable, this obtains the convergence criteria
\begin{equation}
  \parbox{.85\textwidth}{\smallskip\it
  (a)~if $i_j$ is strictly increasing and $x_j\in A_{i_j}$ such that $x_j\to x$
    then $x\in A$; and (b)~if $x\in A$ then there are $x_i\in A_i$ such that
    $x_i\to x$.}
\end{equation}
Also, in that case, a sequence of subsets $A_i$ converges if and only if every
$x_j\in A_{i_j}$ with $x_j\to x$ admits an extension $\hat x_i\in A_i$
(i.e.\ $x_j=\hat x_{i_j}$) such that $\hat x_i\to x$, (e.g.~the limit of
$(-1)^i[0,1]$ does not exist): given convergence, such an $x$ is a cluster point
by Proposition~\ref{pp:limit-points-and-limits}(c), and hence is a limit
point, so by~Proposition~\ref{pp:limit-points-and-limits}(d) there is a
$\hat x_i\in A_i$ such that $\hat x_i\to x$, and replacing $\hat x_{i_j}$ with
$x_{i_j}$ provides the required extension. Conversely, if $x$ is a cluster
point then there is an $x_j\in A_{i_j}$ such that $x_j\to x$, the extension
provides that $x$ is a limit point, and Kuratowski-Painlev\'e convergence
holds. The equivalence of
\begin{equation}\label{eq:KP-sequential-convergence-final}
  \parbox{.85\textwidth}{\smallskip\it
  (a)~$x_j\in A_{i_j}$ and $x_j\to x$ implies $x\in A$; and
  (b)~there is a sequence $x_j\in A_{i_j}$ such that $x_j\to x$.}
\end{equation}
and Kuratowski-Painlev\'e convergence follows because, for any strictly
increasing $i_j$, $x_i\in A_i$ such that $x_i\to x$ gives by restriction
$x_{i_j}\to x$.

\subsection{The Vietoris and co-compact hypertopologies}
\label{section:vietoris-cocompact-hypertopologies}

\smallskip

\begin{figure}[t]
\hfil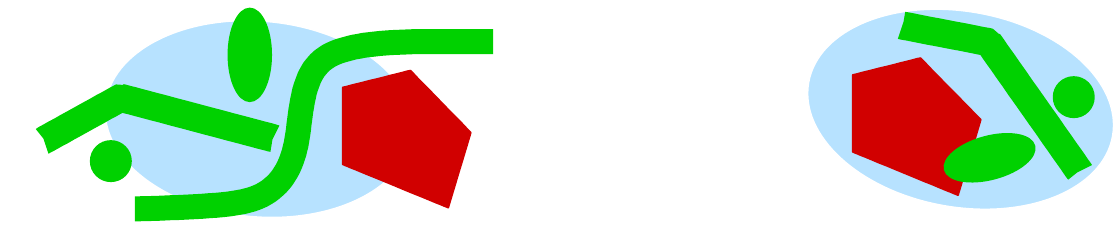\hfil
\renewcommand{\baselinestretch}{.95}
\caption{
Right: a subbasic neighbourhood of $A$ in the upper Vietoris topology is
defined by an open set $U$. The set $A$ is contained in $U$ and the green sets
are in the subbasic neighbourhood are contained in $U$. As $U$ shrinks, every
point in the green sets in drawn to some point in $A$; everything approximable
should be within $A$ if that is closed, corresponding to an ``upper'' type
convergence. At left, in the lower Vietoris topology, the green sets only have
to meet $U$, and shrinking $U$ around a fixed point of $A$ generates
approximations when the green sets also meet $U$, corresponding to a ``lower''
type convergence.
}
\end{figure}

\begin{definition}\label{df:geometric-hypertopologies}
Let $\clX$ be a topological space.
\begin{enumerate}
  \item\label{df:lower-hypertopologies}
    The \emph{lower Vietoris topology [\,lower co-compact topology\,]} on
    $2^\clX$ is generated by the union of
    $\set{A\subseteq\clX}{A\cap U\ne\emptyset}$ over $U\subseteq\clX$ open
    [\,$U\subseteq\clX$ is co-compact\,].
  \item\label{df:upper-hypertopologies}
    The \emph{upper Vietoris topology [\,upper co-compact topology\,]} on
    $2^\clX$ is generated by the union of $\set{A\subseteq\clX}{A\subseteq U}$
    over $U\subset\clX$ open [\,$U\subset\clX$ co-compact\,].
\end{enumerate}
\end{definition}
There are a variety of hypertopologies and the literature is
extensive, e.g.,
\cite{%
  BeerG-1993-2,               %
  IllanesA-NadlerSB-1999-1,   %
  KuratowskiK-1966-1,         %
  KuratowskiK-1968-1,         %
  MichaelE-1951-1,            %
  NadlerSB-1978-1}.
\cite{LucchettiR-PasqualeA-1994-1} presents a systematization and a convenient
summary table. The notation here is not quite standard: it seems that co-compact
is generally defined to mean upper co-compact and the lower co-compact topology
does not appear. The appearance of compactness in these definitions does not
strike as particularly compelling but it is posteriori justified
by~Definition~\ref{df:fell-topology}.

The topologies in Definition~\ref{df:geometric-hypertopologies}
are often referred to as \emph{hit and miss}, since they are generated by
subbases which hit, or set theoretically meet, an open [co-compact] set, and
miss, or are set theoretically contained in, the complement of a closed
[compact] set (for the upper topologies, one can use the alternates
$\set{A}{A\cap K=\emptyset}$ over compact or closed $K$, corresponding to
``miss'').  The \emph{Vietoris topology} or \emph{exponential
  topology}~\cite{KuratowskiK-1966-1} is the join of the upper and lower
Vietoris topologies.

If $\clX$ is Hausdorff then every compact set is closed, every co-compact set
is open, the union of $\set{A\subseteq\clX}{A\subseteq U}$ over co-compact $U$
is contained in the union of that over open $U$, so the Vietoris topologies are
finer than the co-compact topologies. If $\clX$ is a compact Hausdorff space
then the co-compact subsets and the open subsets of $\clX$ coincide, so in that
case the upper Vietoris [lower Vietoris] and the upper co-compact [lower
co-compact] topologies are the same.

\begin{theorem}
\label{th:hyperspace-topology-restriction-natural}
The Vietoris [\,co-compact, assuming $\clY$ is closed\,] topologies on $2^\clX$
are natural with respect to subspaces, i.e.\ if $\clY\subset\clX$ then the
topology of $2^\clY$ as a subspace of $2^\clX$ is the corresponding topology on
$2^\clY$.
\end{theorem}

\begin{proof}
The unions, over open $U\subseteq\clX$, of the left and right sizes of
\begin{equation*}
  2^\clY\cap\set{A\subseteq\clX}{A\cap U\ne\emptyset}=
  \set{A\subseteq\clY}{A\cap(U\cap\clY)\ne\emptyset}.
\end{equation*}
are subbases for the subspace topology on $2^\clY$ and its lower Vietoris
topology, respectively. Similarly,
\begin{equation*}
  2^\clY\cap\set{A\subseteq\clX}{A\subseteq U}=
  \set{A\subseteq\clY}{A\subseteq\clY\cap U}.
\end{equation*}
shows the two two upper Vietoris are the same. For the upper co-compact
topology, the unions of
\begin{equation*}
  2^\clY\cap\set{A\subseteq\clX}{A\cap K\ne\emptyset},
  \qquad
  \set{A\subseteq\clY}{A\cap L\ne\emptyset},
\end{equation*}
over compact $K\subseteq\clX$ and $L\subseteq\clY$ are subbases for the
subspace lower Vietoris topology on $2^\clY$, and the lower Vietoris topology
on $2^\clY$ using the subspace topology on $\clY$. The two collections
correspond for choices of $K$ and $L$: if $\clY$ is closed and $K$ is compact
in $\clX$, then $L=K\cap\clY$ is compact in $\clY$, while if if $L$ is compact
in $\clY$ then it is compact in $\clX$. Similarly, for $K\subseteq\clX$ compact
and $L\subseteq\clY$ compact, the two collections
\begin{equation*}
  2^\clY\cap\set{A\subseteq\clX}{A\cap(\clX\setminus K)\ne\emptyset},
  \qquad
  \clN_\clY(L)\equiv\set{A\subseteq\clY}{A\cap(\clY\setminus L)\ne\emptyset},
\end{equation*}
correspond: if $\clY$ is closed and $K$ is compact in $\clX$, then $K\cap\clY$
is compact in $\clY$ and
$A\cap(\clX\setminus K)=A\cap(\clY\setminus K)
  =A\cap(\clY\setminus(K\cap \clY))$, while if $L$ is compact in $\clY$ then
$L$ is compact in $\clX$ and $A\subseteq\clY$ implies $A\cap(\clX\setminus
L)=A\cap(\clY\setminus L)$.
\end{proof}

Sequences are the target here and so the countability of these hypertopologies
is a focus. As it turns out, the most natural route to sequential convergence,
first countability, is a deeper problem~\cite{BeerG-1993-1}. However, the
(general) topologies for the underlying spaces of numerical analysis are
simple: the usual assumption is at least a second countable locally compact
Hausdorff space (and therefore paracompact). Second countability is stronger
than first and can be relatively easily passed to hypertopologies:

\begin{theorem}\label{th:hypertopology-countable}
If $\clX$ is a second countable [\,and locally compact Hausdorff\,] then the
lower Vietoris [\,upper co-compact\,] topology is second countable.
\end{theorem}

\begin{proof}
Suppose $\clX$ has a countable basis $\clB_0$. For the lower Vietoris
topology, define
$\clN(U)\equiv\set{A\in2^\clX}{A\cap U\ne\emptyset}$.
It suffices to show $\clB\equiv\set{\clN(U)}{\mbox{$U$ is open}}$ and
$\clB'\equiv\set{\clN(U)}{U\in\clB_0}$ are equivalent subbases, the first being
the defining subbase of the lower Vietoris and the second being
countable. Indeed, $\clB'\subseteq\clB$, while if $\clN(U)\in\clB$ and
$A\in\clN(U)$ then let $x\in A$ and choose $V\in\clB$ such that
$x\in V\subseteq U$, and the result follows because
$A\in\clN(V)\subseteq\clN(U)$. 

For the upper co-compact topology, define
$\clN(K)\equiv\set{A\subseteq\clX}{A\cap K=\emptyset}$. Because $\clX$ is
locally compact and Hausdorff, the set $\clB_{00}$ of relatively compact
subsets of $\clB_0$ is a countable basis of $\clX$. It suffices to show
$\clB\equiv\set{\clN(K)}{\mbox{$K$ is compact}}$ and
$\clB'\equiv\set{\clN(\onm{cl}U)}{U\in\clB_{00}}$ are equivalent subbases, the
first being the defining subbase of the upper co-compact topology and the
second being countable. Indeed, $\clB'\subseteq\clB$, while if $K$ is compact
and $A\in\clN(K)$, then, using regularity of $\clX$ and $A$ is closed, choose
finitely many $U_1,\ldots,U_n\in\clB_{00}$ that cover $K$ with closures
contained in $\clX\setminus A$, from which
$A\in\clN(\onm{cl}U_1)\cap\ldots\cap\clN(\onm{cl}U_n)\subseteq\clN(K)$.
\end{proof}

Finding a hypertopology that fits the
criteria~\eqref{eq:evaluation-convergence-3} means a precise understanding of
the convergence defined by those. That can be expressed in terms of Kuratowski
limits:

\begin{lemma}
\label{lm:geometric-hypertopology-convergence}
Let $\clX$  be a topological space.
\begin{enumerate}\alphalist
  \item
    If $X$ is regular and $A_\lambda\to A$ (upper Vietoris) then
    $\onm{cl}A\supseteq\onm{Ls} A_\lambda$.
  \item
    If $\clX$ is locally compact and $A_\lambda\to A$ (upper co-compact) then
    $\onm{cl}A\supseteq\onm{Ls} A_\lambda$.
  \item
    If $\clX$ is compact and $A\supseteq\onm{Ls} A_\lambda$ then $A_\lambda\to
    A$ (upper Vietoris).
  \item
    If $A\supseteq\onm{Ls} A_\lambda$ then $A_\lambda\to A$ (upper co-compact).
  \item
    $A_\lambda\to A$ (lower Vietoris) if and only if 
    $A\subseteq\onm{Li} A_\lambda$.
  \item
    If $\clX$ is Hausdorff and $A\subseteq\onm{Li} A_\lambda$ then
    $A_\lambda\to A$ (lower co-compact).
  \item
    If $\clX$ is compact and $A_\lambda\to A$ (lower co-compact) then
    $A\subseteq\onm{Li} A_\lambda$.
\end{enumerate}
\end{lemma}

\begin{proof}
(a)
Suppose $A_\lambda\to A$ (upper Vietoris), $x$ is a cluster point of
$A_\lambda$, $U\ni x$ is open, and $V\supseteq\onm{cl}A$. By regularity, it
suffices to show $U\cap V\ne\emptyset$. Choose $\lambda^*$ such that
$A_\lambda\subseteq V$ for $\lambda\ge\lambda^*$, choose $\lambda\ge\lambda^*$
such that $A_\lambda\cap U\ne\emptyset$, and note that $U\cap V\supseteq
A_\lambda\cap U$.

(b)
Suppose $A_\lambda\to A$ (upper co-compact). Since $\clX$ is locally compact,
if $x\not\in\onm{cl}A$ then there is a compact neighbourhood $U$ of $x$ such
that $A\cap U=\emptyset$.  So there is a $\lambda^*$ such that $A_\lambda\cap
U=\emptyset$ for all $\lambda\ge\lambda^*$ i.e.\ $x$ is not a cluster point of
$A_\lambda$.

(c)
Suppose $A_\lambda\not\to A$ (upper Vietoris) i.e. there is an open $U$ such
that $A\subseteq U$, and, for all $\lambda^*$ there is $\lambda$ such that
$\lambda\ge\lambda^*$ and $A_\lambda\not\subseteq U$. The set
$\Lambda'=\set{\lambda\in\Lambda}{\mbox{$A_\lambda\not\subseteq U$}}$ is
directed. For each $\lambda\in\Lambda'$ choose $x_\lambda\in A_\lambda$ such
that $x_\lambda\not\in U$. Since $\clX$ is compact, a subnet $x_{\lambda_\mu}$
converges, say to $x$; by Proposition~\ref{pp:limit-points-and-limits}(a)
$x\in\onm{Ls} A_\lambda$. But $x\not\in U$ since $\clX\setminus U$ is closed
and $x_{\lambda_\mu}\not\in U$ from which $x\not\in A$ since $A\subseteq U$.

(d)
Suppose $A_\lambda\not\to A$ (upper co-compact) i.e.\ there is a compact $K$
such that $A\cap K=\emptyset$, and, for all $\lambda^*$ there is $\lambda$ such
that $\lambda\ge\lambda^*$ and $A_\lambda\cap K\ne\emptyset$. The set
$\Lambda'=\set{\lambda\in\Lambda}{\mbox{$A_\lambda\cap K\ne\emptyset$}}$ is
directed. For each $\lambda\in\Lambda'$ choose $x_\lambda\in A_\lambda$ such
that $x_\lambda\in K$. A subnet $x_{\lambda_\mu}$ converges, say to $x\in K$.
So $x$ is an cluster point of $A_{\lambda}$ and $x\not\in A$, from which
$A\not\supseteq\onm{Ls} A_\lambda$.

(e)
If $A_\lambda\to A$ (lower Vietoris) and $x\in A$ and $U\ni x$ is open then
$U\cap A\ne\emptyset$ and there is a $\lambda^*$ such that $U\cap
A_\lambda\ne\emptyset$ for all $\lambda\ge\lambda^*$, so $x$ is a limit
point. The converse is similar.

(f,g)
By~(e), $A_\lambda\to A$ in the lower Vietoris topology. In a Hausdorff space,
every co-compact set is open, so the lower Vietoris topology is finer than the
co-compact topology and therefore $A_\lambda\to A$ in the co-compact topology,
while if $\clX$ is compact there every every open set is co-compact.
\end{proof}

\begin{lemma}
\label{lm:order-theoretic-properties-limits}
Let $A_\lambda$ be a net of subsets of $\clX$.
\begin{enumerate}\alphalist
  \item
    If $A_{\lambda_\mu}$ is a subnet of $A_\lambda$ then
    $\onm{Li} A_{\lambda_\mu}\supseteq\onm{Li} A_\lambda$.
  \item
    If $A=\onm{Li} A_{\lambda_\mu}$ is maximal in the lower limit sets of
    $A_\lambda$ then $A_{\lambda_\mu}\to A$ in the upper co-compact topology.
\end{enumerate}
\end{lemma}

\begin{proof}
(a)
If $A_{\lambda_\mu}$ is a subnet of $A_\lambda$ then $x\in\onm{Li} A_\lambda$
and $U\ni x$ is open then choose $\lambda^*$ such that $A_\lambda\cap           
U\ne\emptyset$ for all $\lambda\ge\lambda^*$. Choose $\mu^*$ so that
$\lambda_{\mu^*}\ge\lambda^*$.  Then $\mu\ge\mu^*$ implies
$\lambda_\mu\ge\lambda^*$ and $A_{\lambda_\mu}\cap U\ne\emptyset$.

(b)
Suppose $A$ is maximal. If $A_{\lambda_\mu}\not\to A$ (upper co-compact) then
there is a compact set $K$ such that $K\cap A=\emptyset$ and such that, for all
$\mu^*$ there is a $\mu>\mu^*$ such that $A_{\lambda_\mu}\cap
K\ne\emptyset$. The set $\set{\mu}{A_{\lambda_\mu}\cap K\ne\emptyset}$ is
directed; choose $x_\mu\in A_{\lambda_\mu}\cap K$ for each such $\mu$. Since
$K$ is compact, a subnet of $x_\mu$ converges, say to $x\in K$.  The limit
inferior of the corresponding subnet of $A_{\lambda_\mu}$ contains $A$ and $x$,
contradicting maximality of $A$ since $A\cap K=\emptyset$.
\end{proof}

\subsection{The Fell topology}
\label{section:fell-topology}

An emphasized in the Appendix, any convergence criteria that respects
subnets does generate a topology: the finest topology such that every net
that satisfies the criteria converges. A convergence is called
\emph{topological} on the equivalence of the apriori criteria and the
convergence in the generated topology. Otherwise, there may be nets that
converge in the generated topology that do not satisfy such apriori
criteria. Such a condition may manifest a logical inequivalence of convergence
statements: convergence in the criteria implies convergence in the topology,
but the converse of that involves some additional context or conditions.

This is further complicated because the generated topology is usually
abstract\,---\,it is an unknown that one is seeking to identify. In that case,
a putative topology may deviate in both directions of logical implication. One
strategy to deal with this is to pair implications that appear to be almost
equivalent and then take the union of their logical predicates, hoping that is
not so restrictive as to eliminate the target application. In case of success,
one arrives at a viable context in which the criteria are matched to a
topology. This is not an exact science. One is sorting through logical
implications that may not be the best. Nothing prevents further work from
establishing a wider context that matches a given convergence criteria to a
larger class of topologies.

Following these ideas, Lemma~\ref{lm:geometric-hypertopology-convergence} seems
to diminish the upper Vietoris topology or lower co-compact topologies, because
Lemma~\ref{lm:geometric-hypertopology-convergence}(c) and
Lemma~\ref{lm:geometric-hypertopology-convergence}(g) depend on compactness
actual.  While it may be reasonable to assume that the ambient spaces of
numerical analysis are locally compact, they are  not usually compact.
In passing, those hypertopologies seem somewhat pathological anyway. For
example,~if $\clX=\bbR^2$ and $A_n=\bbR\times\sset{1/n}$ then $A_n$ does not
converge in the upper Vietoris topology to $\bbR\times\sset{0}$ because $A_n$
is not contained in any neighbourhood of the form $\set{(x,y)}{-1/x<y<1/x}$,
while $A_n$ converges to any bounded set in the lower co-compact topology since
it meets the complement of any compact set.

The condition $A\subseteq\onm{Li} A_\lambda$ already topological to the lower
Vietoris topology, by
Lemma~\ref{lm:geometric-hypertopology-convergence}(e). Aligning
Lemma~\ref{lm:geometric-hypertopology-convergence}(b) and
Lemma~\ref{lm:geometric-hypertopology-convergence}(d), the
condition $A\supseteq\onm{Ls} A_\lambda$ is topological in the context of
closed sets on a locally compact Hausdorff space. By
Theorem~\ref{th:convergence-topology-lattice}(c), the combination of
$A\subseteq\onm{Li} A_\lambda$ and $A\supseteq\onm{Ls} A_\lambda$,
i.e.~Kuratowski-Painlev\'e convergence, is topological to the join of the lower
Vietoris and upper co-compact topologies. In the case that $\clX$ is also second
countable, both the lower Vietoris topology and upper co-compact topologies are
second countable, as is their join, and then sequences suffice, with
convergence the extremely compelling~\eqref{eq:KP-sequential-convergence-final}.

\begin{definition}\label{df:fell-topology}
The \emph{Fell topology} is the join of the upper co-compact topology and the
lower Vietoris topology. $\onm{Fell}(\clX)$ is the set of closed subsets of
$\clX$ with the Fell topology.
\end{definition}

The Fell topology first occurs in~\cite{ChabautyC-1950-1}
and~\cite{FellJMG-1962-1}; see also~\cite{BeerG-1993-1}. The nomenclature is
not standardized. The term Chabauty-Fell is suggested
in~\cite{delaHarpeP-2008-1}. \cite{MatsuzakiK-2017-1} calls the lower Vietoris
topology the Thurston topology, and, with
\cite{CanaryRD-EpsteinDBA-GreenPL-2006-1}, the Fell topology is called the
Chabauty topology, but the latter also use the term geometric topology. The
notation $\onm{Fell}(\clX)$ is not standard.

\begin{theorem}
\label{th:fell-topology-properties}\mbox{}
\begin{enumerate}\alphalist
  \item
    $\onm{Fell}(\clX)$ is compact (for any $\clX$).
  \item
    If $\clX$ is locally compact then $\onm{Fell}(\clX)$ is Hausdorff and
    $A_\lambda\to A$ if and only if $A_\lambda\to A$ (Kuratowski-Painlev\'e).
  \item
    If $\clX$ is locally compact, Hausdorff, and second countable then
    $\onm{Fell}(\clX)$ is second countable.
\end{enumerate}
\end{theorem}

\begin{proof}
(a)
Suppose $A_\alpha=\onm{Li} f_\alpha$ is a (set theoretic) chain of lower limit
sets, where $f_\alpha\colon M_\alpha\to\Lambda$. The co-product 
$M\equiv\bigvee M_\alpha=\set{(\alpha,\mu)}{\mu\in M_\alpha}$ is ordered by
$(\alpha_1,\mu_1)\ge(\alpha_2,\mu_2)$ if $\alpha_1=\alpha_2$ and
$\mu_1\ge\mu_2$ ( $(\alpha_1,\mu_1)$ and $(\alpha_2,\mu_2)$ are incomparable if
$\alpha_1\ne\alpha_2$). With this ordering, $(\alpha,\mu)\mapsto f_\alpha(\mu)$
is a subnet of $A_\lambda$ and $\onm{Li}f\supseteq\bigcup_\alpha A_\alpha$.
Thus every chain has an upper bound, Zorn's lemma provides a maximal lower
limit set, and there is a convergent subnet of $A_\lambda$
by Lemma~\ref{lm:order-theoretic-properties-limits}(b).

(b)
If $A_\lambda\to A$ in $\onm{Fell}(\clX)$ then $A_\lambda\to A$ in both the
upper co-compact and the lower Vietoris topology. By
Lemma~\ref{lm:geometric-hypertopology-convergence},
$\onm{Li}A_\lambda=A=\onm{Ls}A_\lambda$, which establishes that limits are
unique.  If $\clX$ is locally compact then by
Lemma~\ref{lm:geometric-hypertopology-convergence} the upper
co-compact and lower Vietoris topologies separately correspond to the
convergence criteria $A\supseteq\onm{Ls A_\lambda}$ and 
$\onm{Ls A_\lambda}\subseteq A$, respectively. Because of that, the join of
those topologies corresponds to the logical ``and'' of those criteria, and that
by \eqref{eq:KP-sequential-convergence-final} is equivalent to Kuratowski-Painlev\'e
convergence.

(c)
If $\clX$ is second countable locally compact Hausdorff
then Theorem~\ref{th:hypertopology-countable} provides countable bases for both
the lower Vietoris and upper co-compact topologies, and the union of those is a
countable basis for the Fell topology.
\end{proof}

\begin{figure}[t]
\hfil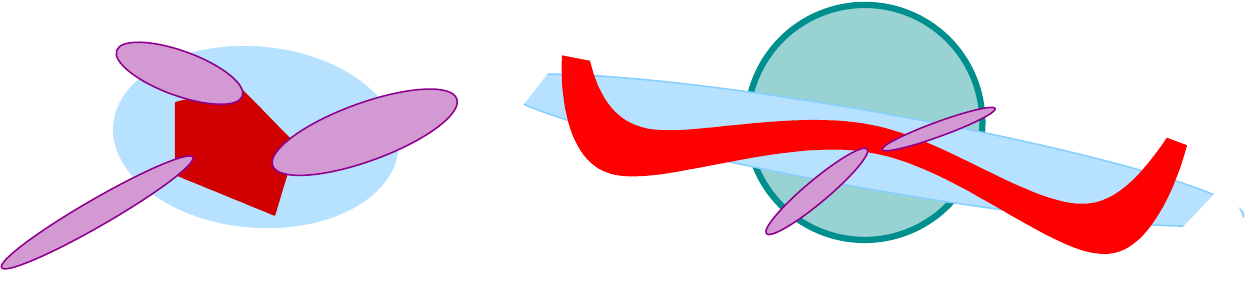\hfil
\renewcommand{\baselinestretch}{.95}
\caption{
A neighbourhood base of a compact set $A$ is pictured top at left. A subset is
inside such a neighbourhood if it is contained in $U$ and makes contact with
each $V_i$. If the set is not compact, pictured at right, then the same has to
occur just inside a compact set $K$, which is a part of the neighbourhood's
definition (the neighbourhood exerts control only within $K$). The effect on
convergence of $A_i$ is that membership in such a neighbourhood must be
established finally, say for some $i\ge N$, but larger $K$ require larger~$N$.
}
\end{figure}

\begin{theorem}
\label{th:fell-topology-base}
\mbox{}
\begin{enumerate}\alphalist
\item
  The collection
  $\set{A\subseteq\clX}{\mbox{$A\cap K\subseteq U$ and $A\cap V_i\ne\emptyset$
  for all $i$}}$
  over all open $U,V_1,\ldots,V_n\subseteq\clX$ and compact $K\subseteq\clX$,
  is a base for $\onm{Fell}(\clX)$.
\item
  Suppose $\clX$ is locally compact and Hausdorff, and let $B\subseteq\clX$ be
  compact subset.  Then the collection
  $\set{A\subseteq\clX}{\mbox{$B\subseteq A$ and 
  $A\subseteq U$ and $A\cap V_i\ne\emptyset$ for all $i$}}$ over all
  relatively compact open $U$, all open subsets $V_1,\ldots,V_n$ is a
  neighbourhood base of $\onm{Fell}(\clX)$ (at $B$).
\end{enumerate}
\end{theorem}

\begin{proof}
(a)
The given collection is the same as
\begin{equation*}
\set{A\subseteq\clX}
 {\mbox{$A\cap K\cap L=\emptyset$ and $A\cap V_i\ne\emptyset$ for all $i$}}
\end{equation*}
over all closed $L$, open $V_i$, and compact~$K$. Since the intersection of a
closed set and a compact set is compact, and since one may take $L=\clX$, this
is the same as the collection
$\set{A\subseteq\clX}
  {\mbox{$A\cap K=\emptyset$ and $A\cap V_i\ne\emptyset$ for all $i$}}$,
which is a base for $\onm{Fell}(\clX)$ by definition of the join of the upper
co-compact and lower Vietoris topologies.

(b)
If $U$ is relatively compact then set $K=\onm{cl}U$ and note that 
$A\cap K\subseteq U$ if and only if $A\subseteq U$, so the collection is an
extraction from the base established in Theorem~\ref{th:fell-topology-base}(a).
Now if $U,V_i$ are open, $K$ is compact, and suppose a compact $B$ such that
$B\cap K\subset U$. Then $B$ and $K\setminus U$ are
disjoint compact sets, so there are disjoint open relatively compact
$U'\supseteq B$ and $U''\supseteq K\setminus U$. If $A\subseteq U'$ then 
$A\cap K\subseteq U'\cap K\subseteq K\cap(\clX\setminus U'')\subseteq
  K\cap\bigl(\clX\setminus(K\setminus U)\bigr)=K\cap U\subseteq U$.
\end{proof}

\section{Partial maps}
\label{section:partial-maps}

A \emph{partial function} on $\clX$ is a function with domain a (usually, but
not necessarily proper) subset of $\clX$. Since partial maps have by
definition domains which are subsets, and since some topologies on partial maps
are defined by replacing the partial maps with graphs, their topologies are
strongly related to topologies on spaces of subsets of a topological
space.

\begin{definition}
\label{df:geometric-topologies-partial-maps}
Let $\clX$ and $\clY$ be topological spaces.
\begin{enumerate}\alphalist
  \item
    The topologies on the set of partial maps from $\clX$ to $\clY$ with the
    same names
    as Definition~\ref{df:geometric-hypertopologies} are
    the corresponding topologies obtained after identifying maps with
    their graphs.
  \item
    The \emph{compact-open} topology on the set of partial maps from $\clX$ to
    $\clY$ is the topology with subbase sets of the form
    $\clN(K,U)\equiv\set{f\colon B\to\clY}{f(K)\subseteq U}$, where
    $K\subseteq\clX$ is compact and $U\subseteq\clY$ is open.
  \item
    The \emph{Back topology} is the join of the compact-open topology and
    initial topology obtained from assignment of partial maps to
    domains with the lower Vietoris topology.
\end{enumerate}
\end{definition}

The first to consider spaces of partial maps was \cite{KuratowskiK-1955-1}.  A
convergence criteria occurs even at the beginning in the following form: a
sequence $f_n$ with domains $A_n$ convergences to $f$ with domain $A$ if and
only if $A_n$ converges to $A$ in the Hausdorff distance and $x_n\rightarrow
x$, and $x\in A$ implies $f_n(x_n)\rightarrow f(x)$. The Back topology is
from~\cite{BackK-1986-1}. A recent reference for topologies on the space of
partial maps is~\cite{BeerG-CasertaA-DiMaioG-LucchettiR-2014-1}.  The term
\emph{generalized compact-open} is a synonym for the Back topology, and
\emph{compact-open} may not be as used here but rather may be synonymous with
Back topology.

If $\clX$ and $\clY$ are second countable locally compact Hausdorff spaces then
$\clX\times\clY$ is also that. Let $\onm{Fell}(\clX,\clY)$ be the set of
such with the subspace topology obtained by identifying partial maps with
their graphs. Subspaces of second countable [Hausdorff] spaces are second
countable [Hausdorff] (\cite{WillardS-1970-1}, Theorems~13.8 and~16.2 and so
$\onm{Fell}(\clX,\clY)$ is second countable and Hausdorff if $\clX$ and
$\clY$ are.

\begin{theorem}
\label{th:fell-convergence-partial-function}
A sequence $(f_i\colon A_i\to\clY)\in \onm{Fell}(\clX,\clY)$ converges
  to $(f\colon A\to\clY)\in\onm{Fell}(\clX,\clY)$ if and only if
\begin{enumerate}\alphalist
  \item
    for all strictly increasing $i_j$, if $x_j\in A_{i_j}$ and both $x_j$ and
    $f_{i_j}(x_j)$ converge, then $\lim x_j\in A$ and $f(\lim x_j)=\lim
    f_{i_j}(x_j)$; and
  \item
    for all $x\in A$ there is an $x_i\in A_i$ such that $x_i\to x$ and $f(x_i)$
    converges.
\end{enumerate}
If in addition $\clY$ is compact, then $f_i$ converges if and only if
\begin{enumerate}\alphalist\setcounter{enumi}{2}
  \item
    for all strictly increasing $i_j$, if $x_j\in A_{i_j}$ and $x_j$ converges,
    then $\lim x_j\in A$ and $f(\lim x_j)=\lim f_{i_j}(x_j)$; and
  \item
    for all $x\in A$ there is an $x_i\in A_i$ such that $x_i\to x$.
\end{enumerate}
\end{theorem}

\begin{proof}
Suppose $f_i\to f$ in $\onm{Fell}(\clX,\clY)$. If $i_j$ is strictly increasing,
$x_j\in A_{i_j}$, $x_j\to x$, and $f_{i_j}(x_j)\to y$, then
$\bigl(x_j,f_{i_j}(x_j)\bigr)\in\onm{graph} f_{i_j}$ and
$\bigl(x_j,f(x_{i_j})\bigr)\to (x,y)$, so $(x,y)\in\onm{graph}f$ and $\lim
f_{i_j}(x_j)=y=f(x)=f(\lim x_j)$. Also, if $x\in A$ then
$\bigl(x,f(x)\bigr)\in\onm{graph} f$ and there is a
$(x_i,y_i)\in\onm{graph}f_i$ such that $(x_i,y_i)\to (x,y)$, implying both
$x_i\to x$ and $f(x_i)$ converges.

Conversely, suppose~(a) and~(b).  If $i_j$ is strictly increasing and
$(x_j,y_j)\in\onm{graph}f_{i_j}$ and $(x_j,y_j)\to(x,y)$, then both $x_j$ and
$y_j=f_{i_j}(x_j)$ converge, so $x=\lim x_j\in A$ and $y=\lim
f_{i_j}(x_j)=f(\lim x_j)=f(x)$, i.e., $(x,y)\in\onm{graph}f$. If
$(x,y)\in\onm{graph}f$ then $x\in A$ and there is an $x_i\in A_i$ such that
$x_i\to x$ and $f(x_i)$ converges. By (a), $y=f(x)=f(\lim x_i)=\lim f_i(x_i)$
so $\bigl(x_i, f(x_i)\bigr)\in\onm{graph}f_i$ and $\bigl(x_i, f(x_i)\bigr)\to
(x,y)$, and $f_i\to f$ in $\onm{Fell}(\clX,\clY)$ follows
from~\eqref{eq:KP-sequential-convergence-final}.

Irrespective of whether $\clY$ is compact or not, suppose~(c) and~(d). Then~(a)
is immediate from~(c), while for~(b), if $x\in A$ then there is a $x_i\in A_i$
such that $x_i\to x$, and~(c) implies that $f(x_i)$ converges.  Conversely,
suppose~(a) and~(b). Then~(d) is immediate, while for~(c), supposing $\clY$ is
compact, let $i_j$ be strictly increasing, and let $x_j\in A_j$ be such that
$x_j\to x$. A subsequence of $f(x_{i_j})$ has a convergence subsequence, so
by~(a), $x\in A$ and that subsequence converges to $f(x)$. Since every
subsequence of $f(x_{i_j})$ has a subsequence that converges to $f(x)$,
\cite{WillardS-1970-1} (Exercise~11D) implies that $f(x_{i_j})$ converges, to
$f(x)$.
\end{proof}

The graph of $\tanh(nx)$ converges to the union
$\set{(x,1)}{x\ge 0}\cup\set{(x,-1)}{x\le 0}$, which is not the graph of a
function, so the set of graphs is not necessarily closed in the Fell topology
in the graph space. Thus $\onm{Fell}(\clX,\clY)$ is not necessarily compact
even though the Fell topology on the closed subsets is compact.

For comparison, the result for the far more well-known compact-open topology
follows. Note that the convergence criteria for the forward implication is 
essentially different from the reverse, even given local compactness.

\begin{proposition}\label{pp:compact-open-convergence}{
Let $\clX$ and $\clY$ be topological spaces, and let $f_\lambda\colon
A_\lambda\to\clY$ be a net of continuous maps.
\begin{enumerate}\alphalist
  \item
    Suppose $\clX$ be a locally compact. If $f_\lambda$ converges in the
    compact-open topology then $x_\mu\to x$, such that $x_\mu\in
    A_{\lambda_\mu}$ and $x\in A$, implies $f_{\lambda_\mu}(x_\mu)\to f(x)$.
  \item
    Suppose $\clX$ is locally compact and Hausdorff.  If $x_\mu\in
    A_{\lambda_\mu}$ and $x_\mu\to x$ implies both $x\in A$ and
    $f_{\lambda_\mu}(x_\mu)\to f(x)$, then $f_\lambda$ converges to $f$ in the
    compact-open topology.
\end{enumerate}}
\end{proposition}

\begin{proof}
(a)
Let $x_\mu\in A_{\lambda_\mu}$ be such that $x_\mu\to x\in A$, and let
$U\ni f(x)$ be open. Choose a compact neighbourhood $K\ni x$ such that
$f(K)\subseteq U$. Choose $\lambda^*$ such that $f_\lambda(K)\subseteq U$ for
$\lambda\ge\lambda^*$. Choose $\mu_1^*$ such that $x_\mu\in K$ for
$\mu\ge\mu_1^*$. Choose $\mu_2$ such that $\lambda_\mu\ge\lambda^*$ for
$\mu\ge\mu_2^*$. Choose $\mu^*$ such that $\mu^*\ge\mu_1^*$ and
$\mu^*\ge\mu_2^*$. Then $\mu\ge\mu^*$ implies both $x_\mu\in K$ and
$f_{\lambda_\mu}(K)\subseteq U$, from which $f(x_{\lambda_\mu})\in U$. This
shows $f_{\lambda_\mu}(x_\mu)\to f(x)$ since $U$ was an arbitrary open
set containing $f(x)$.

(b)
Assume that $f_\lambda$ does not converge in the compact-open topology.
Assumed, then, is a compact subset $K$ of $\clX$, and an open $U$ with
$f(K)\subseteq U$, such that, for all $\lambda^*$ there is a
$\lambda\ge\lambda^*$ such that $f_\lambda(K)\not\subseteq U$, 
implying by Lemma~\ref{lm:cofinal-implies-directed} that 
$\Lambda'=\set{\lambda}{f_\lambda(K)\not\subseteq U}$ is
directed. For each $\lambda\in\Lambda'$,
pick $x_\lambda\in K$ such that $f_\lambda(x_\lambda)\not\in U$. There is a
convergent subnet $x_{\lambda_\mu}\ $ of $x_\lambda$, say
$x_{\lambda_\mu}\to x$. Then $x\not\in A$, or $x\in A$ and
$f_{\lambda_\mu}(x_\mu)$ does not converge to $f(x)$, where 
$x_\mu\equiv x_{\lambda_\mu}$, because $K$ is closed (since $\clX$ is Hausdorff)
implies $x\in K$, from which $f(x)\in U$, whereas
$f_{\lambda_\mu}(x_\mu)\not\in U$.
\end{proof}

\section{Conclusions}
\label{section:conclusions}

In the numerical analysis of ordinary differential equations, we deal with mesh
functions of a real variable, with values in $\bbR^n$. For a partial
differential equation the domain may be such as an open subset of a manifold.
Such domains are almost universally second countable locally compact Hausdorff
spaces.  And then, with no further assumptions,
Theorem~\ref{th:fell-convergence-partial-function} provides that the natural
invariant convergence notion~\eqref{eq:evaluation-convergence-3} exactly
corresponds to the well-studied Fell topology on the space of numerical
approximations.

Fell-convergence is essentially equivalent to verifying machine
convergence over varying start and end times and with a varying grid.  The
closed sets of the Fell topology are logically accessible via the convergence
criteria; the open sets via the explicit and easily visualized local base
provided by Theorem~\ref{th:fell-topology-base}. The development operates at
the topological level\,---\,it does not presuppose any smooth differentiable
structure.

\section*{Appendix: topologies from convergence criteria}

A \emph{directed set} is a set $\Lambda$ with a \emph{directed preorder}:
a relation $\ge$ which satisfies: (1)~\emph{reflexive}: $\lambda\ge\lambda$;
(2)~\emph{transitive}: $\lambda_1\ge\lambda_2$ and $\lambda_2\ge\lambda_3$
implies $\lambda_1\ge\lambda_3$; and (3)~\emph{upper bounds}: for all
$\lambda_1$ and $\lambda_2$ there is a $\lambda^*$ such that
$\lambda^*\ge\lambda_1$ and $\lambda^*\ge\lambda_2$. A property $P(\lambda)$ is
\emph{final} if there is a $\lambda^*\in\Lambda$ such that $P(\lambda)$ is
true for all $\lambda\ge\lambda^*$. A property $P(\lambda)$ is \emph{cofinal}
if for all $\lambda^*\in\Lambda$ there is a $\lambda\in\Lambda$ such that
$P(\lambda)$ is true for $\lambda\ge\lambda^*$. A \emph{net} is a function
$x_\lambda$ on a directed set with values in a topological space. $x_\lambda$
\emph{converges} to $x$ if for all neighbourhoods $U\ni x$ there is a
$\lambda^*$ such that $x_\lambda\in U$ for all $\lambda>\lambda^*$, i.e., if
$x_\lambda$ is finally in any neighbourhood of $x$. A \emph{subnet} of
$x_\lambda$ is a net $x_{\lambda_\mu}$ where $\lambda_\mu$
is map from a directed set $M$ to $\Lambda$ which satisfies (1)~\emph{monotone}:
$\mu_1\ge\mu_2$ implies $\lambda_{\mu_1}\ge\lambda_{\mu_2}$; and
(2)~\emph{cofinal}: for all $\lambda\in\Lambda$ there is an $\mu\in M$ such that
$\lambda_\mu\ge\lambda$. $x\in\clX$ is a \emph{cluster point} of~$x_\lambda$
if for all neighbourhoods $U\ni x$ and all~$\lambda^*$ there is
a~$\lambda\ge\lambda^*$ such that $x_\lambda\in U$, i.e., $x_\lambda$ is
cofinally in any neighbourhood of~$x$. The set of cluster points of~$x_\lambda$
will be denoted by $(x_\lambda)^\infty$.

The textbook \cite{RundeV-2005-1} disposes of the monotone condition on
subnets, but imposes the stronger notion of finality: for all
$\lambda^*\in\Lambda$ there is an $\mu^*\in M$ such that
$\lambda_\mu\ge\lambda^*$ whenever $\mu\ge\mu^*$.

\begin{lemma}\label{lm:cofinal-implies-directed}
Suppose $\Lambda$ is directed and $\Lambda'\subseteq\lambda$ is cofinal, i.e,
for all $\lambda\in\Lambda$ there is a $\lambda'\ge\lambda$ such that
$\lambda'\in\Lambda'$. Then $\Lambda'$ is directed.
\end{lemma}

\begin{proof}
If $\lambda'_1,\lambda'_2\in\Lambda'$ then choose $\lambda_3\in\Lambda$
such that $\lambda_3\ge\lambda'_1$ and $\lambda_3\ge\lambda'_2$. Choose
$\lambda'_3\in\Lambda'$ such that $\lambda'_3\ge \lambda_3$. This suffices
because, by transitivity, $\lambda'_3\ge\lambda'_1$ and
$\lambda'_3\ge\lambda'_2$.
\end{proof}

\begin{lemma}\label{lm:cluster-point-iff-subnet}
(Theorem 11.5 of \cite{WillardS-1970-1}).
A net $x_\lambda$ has a cluster point $x$ if and only if it has a subnet
that converges to $x$.
\end{lemma}

\begin{proof} If $x$ is a cluster point of $x_\lambda$
then $\sset{(\lambda, U)}$ such that $U$ is a neighbourhood of $x$ and
$x_\lambda\in U$  is directed by
$(\lambda_1,U_1)\le(\lambda_2, U_2)\Leftrightarrow\lambda_1\le\lambda_2$ and
$U_1\subseteq U_2$ and $(\lambda,U)\mapsto \lambda$ and defines a subnet that
converges to $x$. Conversely, given a subnet $x_{\lambda_\mu}$, a neighbourhood
$U\ni x$, and a $\lambda^*$, choose $\mu$ such that $\lambda_\mu>\lambda^*$ and
$x_{\lambda_\mu}\in U$.
\end{proof}

The standard result regarding convergence and topologies, which
in~\cite{WillardS-1970-1} is relegated to an exercise, is
Theorem~\ref{th:convergence-topology}.

\begin{theorem}\label{th:convergence-topology}
Net convergence on a topological space has the following properties:
\begin{enumerate}\alphalist
  \item
    the constant net $x_\lambda=x$ converges to $x$; and
  \item
    if $x_\lambda$ converges to $x$ then every subnet of $x_\lambda$ also
    converges to $x$; and
  \item
    if $x_\lambda$ converges to $x$ and $x_{\lambda\mu}$ converges to
    $x_\lambda$ for each fixed $\lambda$ and a lexicographic ordering on the
    two character words $\lambda\mu$, then $x_{\lambda\mu}$ has a subnet which
    converges to $x$; and
  \item
    if every subnet of $x_\lambda$ has a subnet that converges to $x$,
    then $x_\lambda$ converges to $x$.
\end{enumerate}
Conversely, given a convergence criteria satisfying~(a)--(c), $\onm{cl}
E\equiv\set{\lim x_\lambda}{x_\lambda\in E}$ is a Kuratowski closure defining a
topology (the \emph{convergence topology}) in which a net that satisfies the
convergence criteria also converges in the topology.  If the convergence also
satisfies~(d) then every net which converges in the topology also satisfies the
convergence criteria.
\end{theorem}

\begin{proof}
Net convergence on topological spaces does satisfy~(a)--(d):
\begin{enumerate}\alphalist
  \item
    The constant net $x_\lambda=x$ converges because if $U\ni x$ is
    open then any $\lambda^*$ provides $x_\lambda\in U$ for all
    $\lambda\ge\lambda^*$.
  \item
    Suppose $x_{\lambda_\mu}$ is a subnet of $x_\lambda\to
    x$. If $U\ni x$ is open then choose $\lambda^*$ so that
    $\lambda\ge\lambda^*$ implies $x_\lambda\in U$. By the definition
    of a subnet, there is a $\mu^*$ such that $\mu\ge\mu^*$ implies
    $\lambda_\mu\ge\lambda^*$, so $\mu\ge\mu^*$ implies
    $x_{\lambda_\mu}\in U$.
  \item
    By Lemma~\ref{lm:cluster-point-iff-subnet}, it suffices to show that
    $x$ is a cluster point of $x_{\lambda\mu}$. Suppose $U\ni x$ is open,
    $\lambda^*\in\Lambda$ and $\mu^*\in M_{\lambda^*}$. Since
    $x_\lambda\rightarrow x$, there is a $\lambda>\lambda^*$ such that
    $x_\lambda\in U$.  Since $x_{\lambda\mu}\rightarrow x_\lambda$ in $\mu\in
    M_\lambda$, and since $x_\lambda\in U$, there is a $\mu\in M_{\lambda}$
    such that $x_{\lambda\mu}\in U$ ($\mu^*$ is irrelevant because
    $\lambda>\lambda^*$ and the ordering is lexicographic).
  \item
    If $x_\lambda$ does not converge to $x$ then there is an open $U\ni x$ such
    that, for all $\lambda^*$ there is a $\lambda\ge\lambda^*$ with
    $x_\lambda\not\in U$.  Then $M=\set{\mu\in\Lambda}{x_\mu\not\in U}$ is
    directed: if $\mu_1,\mu_2\in M$ then choose $\lambda^*>\mu_1$ and
    $\lambda^*>\mu_2$, and then there is a $\mu>\lambda^*$ such that
    $x_\mu\not\in U$, so $\mu\in M$ and $\mu\ge\lambda^*\ge\mu_1$ and
    $\mu\ge\lambda^*\ge\mu_2$. Clearly, the restriction of $x_\lambda$ to $M$ is
    a subnet which has no subnet that converges to $x$.
\end{enumerate}

Assuming~(a) and~(c), $E\mapsto\onm{cl}E$ is a closure operation, and so
provides a topology with closed sets exactly those $E$ such that $\onm{cl}E=E$,
as follows:
\begin{itemize}
  \item[--]
    $\onm{cl}\emptyset=\emptyset$: otherwise, after choosing
    $x\in\onm{cl}\emptyset$ there is an $x_\lambda\to x$ with
    $x_\lambda\in\emptyset$, which is impossible.
  \item[--]
    $E\subseteq\onm{cl} E$: if $x\in E$ then any constant net is a net in
    $E$ converging to $x$, so $x\in\onm{cl} E$.
  \item[--]
    $\onm{cl}\onm{cl} E=E$: By the above,
    $E\subseteq\onm{cl} E\subseteq\onm{cl}\onm{cl} E$. On the other hand,
    suppose $x_\lambda\to x$ with $x_\lambda\in\onm{cl} E$ and
    choose $x_{\lambda\mu}$ such that $\lim_\mu
    x_{\lambda\mu}=x_\lambda$ with $x_{\lambda\mu}\in E$. Then
    $x\in\onm{cl} E$ because there is a subnet of $x_{\lambda\mu}$ in
    the lexicographic ordering such that 
    $x_{\lambda\mu}\to x$.
  \item[--]
    If $x\in\onm{cl}A$, then there is a $x_\lambda\rightarrow x$ with
    $x_\lambda\in A$. Since $x_\lambda\in A\cup B$ also, this implies
    $x_\lambda\in\onm{cl}(A\cup B)$. Similarly
    $\onm{cl}B\subseteq\onm{cl}(A\cup B)$ so
    $\onm{cl}A\cup\onm{cl}B\subseteq\onm{cl}(A\cup B)$. Conversely, if
    $x\in\onm{cl}(A\cup B)$ then there is a net $x_\lambda\to x$ with
    $x_\lambda\in A$ or $x_\lambda\in B$. One of $\set{\lambda}{x_\lambda\in
    A}$ or $\set{\lambda}{x_\lambda\in B}$ is directed, or else there would be
    $\lambda_1^A$ and $\lambda_2^A$ with $x_{\lambda^A_1}\in A$,
    $_{\lambda^A_2}\in A$ and $x_\lambda\not\in A$ for all
    $\lambda\ge\lambda^A_1$ and $\lambda\ge\lambda^A_2$, and similarly with $A$
    and $B$ exchanged. But then choosing $\lambda_3$ such that all of
    $\lambda_3\ge\lambda^A_1,\lambda^A_2,\lambda^B_1,\lambda^B_2$ implies
    $x_\lambda\not\in A$ and $x_\lambda\not\in B$, a contradiction. This
    defines a subnet either in $A$ or $B$ which converges to~$x$
    by~(c), so either $x\in\onm{cl}A$ or
    $x\in\onm{cl}B$.\smallskip
\end{itemize}

Since the \emph{closed} sets are more directly defined by the convergence
topology than \emph{open} sets, it is best to directly use the following: in a
topological space, $x_\lambda\to x$ is false if and only if there is a closed
set $K$ such that $x\not\in K$ and, for all $\lambda^*$ there is a
$\lambda\ge\lambda^*$ such that $x_\lambda\in K$. One then shows that
$x_\lambda\to x$ (convergence topology) is false if and only if $x_\lambda\to
x$ (convergence criterion) is false. Note that, under~(a)--(c), the closed sets
defined by
\begin{equation*}
  \mbox{$K$ closed iff $x_\lambda\in K$ and $x_\lambda\rightarrow x$ 
        implies $x\in K$}
\end{equation*}
are exactly the same as the closed sets defined using the topology, by
\begin{equation*}
  \mbox{$K$ closed iff $\onm{cl}K=K$}.
\end{equation*}

Assume~(a)--(c),
and suppose that $x_\lambda\not\to x$ (topology), so there is a
closed set $K$ as in the statement just above. Then $\set{\lambda}{x_\lambda\in
K}$ is directed and so is a subnet. If this subnet converges to $x$ (criteria)
then $x\in K$, a contradiction. So $x_\lambda$ cannot converge to $x$
(criteria) because $x_\lambda$ has a subnet which does not converge
to $x$ (criteria).

Assume~(a)--(d), and suppose that $x_\lambda\not\to x$ (criteria).  Then there
is a subnet $x_{\lambda_\mu}$ of $x_\lambda$, every subnet of which does not
converge to $x$ (criteria). So defining $K=\onm{cl}{\sset{x_{\lambda_\mu}}}$,
it follows that $x\not\in K$. Given any $\lambda^*$, there is a $\mu$ such that
$\lambda_\mu\ge\lambda^*$, and, setting $\lambda=\lambda_\mu$ this provides a
$\lambda\ge\lambda^*$ such that $x_{\lambda}\in K$, from which
$x_\lambda\not\to x$ (convergence topology).
\end{proof}

\begin{lemma}\label{lm:subbase-convergence}
Let $\clC$ be a subbase for a topology $\tau$. Then $x_\lambda\rightarrow
x$ if and only if, for all $V\in\clC$, there is a $\lambda^*$ such that
$x_\lambda\in V$ for all $\lambda\ge\lambda^*$.
\end{lemma}

\begin{proof} 
If $V\in\clC$ then $V$ is open so $x_\lambda$ is finally in $V$ by definition
of net convergence. For the converse, every open $V$ is the finite intersection
of sets $V_i\in\clC$.  For each $i$ choose $\lambda_i^*$ such that
$x_\lambda\in V_i$ for all $\lambda\ge\lambda_i^*$. The upper bound property of
directed sets and transitivity provided an upper bound $\lambda^*$ for all
$\lambda_i^*$. If $\lambda\ge\lambda^*$ then $x_\lambda\in V_i$ for all $i$ and
hence $x_\lambda\in V$.
\end{proof}

\begin{definition}\label{df:weak-topology}
Let $X$ be a set suppose $f_\alpha\colon\clX\to\clY_\alpha$ are maps,
where $Y_\alpha$ are topological spaces. The \emph{weak topology defined by
$f_\alpha$} is the topology with subbase
$\bigcup_{\alpha}\set{f_\alpha^{-1}(V)}{\mbox{$V\subseteq Y_\alpha$ is open}}$.
\end{definition}

\begin{proposition}\label{pp:weak-topology-convergence}
A net $x_\lambda$ converges to $x$ in the weak topology defined by
$f_\alpha\colon\clX\to\clY_\alpha$ if and only if
$f_\alpha(x_\lambda)\to f(x)$ for all $\alpha$.
\end{proposition}

\begin{proof} 
By Lemma~\ref{lm:subbase-convergence}, $x_\lambda\to x$ if and only if for all
$V\subseteq Y_\alpha$ there is a $\lambda^*$ such that $\lambda\ge\lambda^*$
implies $x_\lambda\in f_\alpha^{-1}(V)$, and the latter is equivalent to
$f_\alpha(x_\lambda)\in V$.
\end{proof}

There are a variety of well-known function-space topologies where a central aim
seems to be the capture of some particular notion of convergence, for which the
primary definition of the topology does not have an immediately obvious
relationship. The topology for distributional test function spaces,
the weak and Whitney fine topologies on differentiable maps between manifolds,
and even the familiar compact-open topology are examples.

Theorem~\ref{th:convergence-topology} suggests that, to define a topology
from a convergence, one should verify all of
Theorem~\ref{th:convergence-topology}(a--d).
In case of success the result is topological: not only is the topology is
obtained, but also \emph{a complete characterization of all convergent nets.}
If the aim is just to define a topology, then~(b) alone suffices: define
$A\subset X$ closed if $x_\lambda\in A$ and $x_\lambda\to x$ implies $x\in
A$. Within this definition, it is obvious then $\emptyset$ and $X$ are closed,
and that arbitrary intersections of closed sets are closed.  If $A$ and $B$ are
closed, $x_\lambda\in A\cup B$, and $x_\lambda\to x$, then $x_\lambda\in A$ or
$x_\lambda\in B$ (depending on $\lambda$). One of 
$\set{\lambda}{x_\lambda\in A}$ or $\set{\lambda}{x_\lambda\in B}$ is directed,
or else there would be $\lambda_1^A$ and $\lambda_2^A$ with $x_{\lambda^A_1}\in
A$, $_{\lambda^A_2}\in A$ and $x_\lambda\not\in A$ for all
$\lambda\ge\lambda^A_1$ and $\lambda\ge\lambda^A_2$, and analogously with
$B$. But then choosing $\lambda_3$ such that all of
$\lambda_3\ge\lambda^A_1,\lambda^A_2,\lambda^B_1,\lambda^B_2$ implies
$x_{\lambda_3}\not\in A$ and $x_{\lambda_3}\not\in B$, a contradiction. Thus
there is a subnet either in $A$ or $B$, and that converges to~$x$
by Theorem~\ref{th:convergence-topology}(b), so either $x\in A$ or $x\in B$ as
both those are closed. If a convergence criteria $\gamma$ satisfies~(b) then
the topology just described will be the \emph{topology $\gamma^\dagger$
generated by $\gamma$}, or just the convergence topology. If such a $\gamma$
also satisfies~(a) then it is $T_1$, because every one point set closed.

In such a topology, there are (at least) two notions of convergence: the given
one and net convergence in the topology itself. The first will be referred to
as \emph{generating}, while the
second will be referred to as \emph{topological}.  If $x_\lambda\rightarrow
x$ (generating) then $x_\lambda\rightarrow x$ (topological): Suppose
$x_\lambda\rightarrow x$ (generating) and $x_\lambda\not\rightarrow x$
(topological). Then there is an open neighbourhood $U\ni x$ such that, for all
$\lambda^*$ there is a $\lambda\ge\lambda^*$ such that $x_\lambda\not\in U$.
It follows that $\tilde\Lambda=\set{\lambda\in\Lambda}{x_\lambda\not\in U}$ is
directed and that $x_\lambda$ with $\lambda\in\tilde\Lambda$ is a subnet of
$x_\lambda$ which has no element in $U$.  However
by Theorem~\ref{th:convergence-topology}(b) this subnet converges to $x$
(generating), and $X\setminus U$ is closed, so $x\in X\setminus U$,
contradicting $x\in U$.  Given a convergence
criteria, the convergence topology is the finest topology such that the nets in
the convergence criteria converge in the topology: Suppose $\tau$ is such a
topology and $A$ is closed in $\tau$, and $x_\lambda\rightarrow x$ (generating)
with $x_\lambda\in A$. Then $x_\lambda\rightarrow x$ (in $\tau$) so $x\in A$,
and hence $A$ is closed in the convergence topology.

The point of all this is to facilitate the simple and transparent
identification of topologies from convergence criteria: the topology
$\gamma^\dagger$ is to be thought of as the most exact carrier of the
convergence criteria $\gamma$. There would not be that much gained from that
alone. Interestingly, this set-up is operationally effective, because the
principle topological notions are actually captured in the usual way by the
\emph{generating convergence.}

\begin{theorem}\label{th:preconvergence-suffices-closed-continuous}
Let $x_\lambda\rightarrow x$ be a pre-topological convergence criteria.
\begin{enumerate}
  \item
    $A\subset X$ is closed if and only if $x_\lambda\in A$ and
    $x_\lambda\rightarrow x$ (generating) implies $x\in A$.
  \item
    $U\subset X$ is open if and only if, for all $x_\lambda\to x$
    (generating) such that $x\in U$, there is an $\lambda^*$ such that
    $x_\lambda\in U$ whenever $\lambda\ge\lambda^*$.
  \item
    $f\colon X\rightarrow Y$ is continuous if and only if
    $\lim_\lambda f(x_\lambda)=f(x)$ for all nets $x_\lambda$
    such that $x_\lambda\rightarrow x$  (generating).
\end{enumerate}
\end{theorem}

\begin{proof} 
The first statement is by definition of the closed sets in the convergence
topology. For the next two statements, it suffices to show the converse, i.e.,
it suffices to show that generating convergence suffices.  For the second
statement, if $x\in U$ implies that every net $x_\lambda\to x$ (generating) is
eventually in $U$, then $x_\lambda\in X\setminus U$ and $x_\lambda\to x$
(generating) implies $x\in X\setminus U$ and hence $X\setminus U$ is closed, or
else $x_\lambda$ is both in $X$ and $X\setminus U$ for large enough $\lambda$.
For the third statement, suppose that $\lim_\lambda f(x_\lambda)=f(x)$ for all
nets $x_\lambda\to x$ (generating). If $B\subseteq Y$ is closed and
$x_\lambda\in f^{-1}(B)$ with $x_\lambda\rightarrow x$ (generating), then
$f(x)=f(\lim_\lambda x_\lambda)=\lim_\lambda f(x_\lambda)\in B$ so 
$x\in f^{-1}(B)$. This shows that $f$ if continuous because $f^{-1}(B)$ is
closed whenever $B$ is.
\end{proof}

Incidentally, Theorem~\ref{th:preconvergence-suffices-closed-continuous}
explains why sequences suffice for continuity of linear maps on the test
function spaces of distribution theory~(\cite{RudinW-1973-1}, Theorem~6.6). The
restriction to \emph{linear} maps arises from the issue of local convexity
referred to at the beginning: sequences do not suffice in the general for that
topology.

One has to exercise care, because of the loose relationship between a
convergence criteria and the topology it generates. For example, an element of
$X$ may be approximable from $A\subset X$ using nets in the convergence
topology, but inaccessible from $A$ via the generating convergence. It would be
an easy error to assert the existence of a net $x_\lambda$ such that
$x_\lambda\rightarrow x$ (generating) from the statement $x\in\onm{cl}A$. In
fact, there is the following: let $X=\bbR$ and use the convergence criteria
$x_\lambda\rightarrow x$ if $|x-x_\lambda|\le1$ for all $\lambda$. If $A=[0,1]$
then $\onm{pcl}A=[-1,2]$ while $\onm{cl}A=\bbR$, where $\onm{pcl}A$, or the
pre-closure, denotes the limits of all convergent nets in~$A$. The pre-closure
is not necessarily a closed set because it only contains those limit but not
necessarily limits of those limits.

If $X$ and $Y$ are topological spaces generated by convergence criteria, then
the topology generated by the product criteria $(x_\lambda,y_\lambda)\to(x,y)$
if $x_\lambda\to x$ (generating) and $y_\lambda\to y$ (generating) may be
strictly finer than the product topology, which is after all the coarsest
topology with continuous projections to the factors. Indeed, the criteria
$x_\gamma\to x$ if $|x_\gamma-x|\to 0$ through powers to $1/2$, generates a
topology on $\bbR$ strictly finer than the usual. In the product topology of
two copies of such, an open line segment with irrational slope does not contain
any nonconstant net in the product criteria hence is closed in the topology so
generated. However, open intervals are open in the factor topologies, and their
product is open in the product topology. Therefore such irrational sloped open
segments are not closed in the product topology because they do not contain the
endpoints which are in their closure; the product topology is strictly more
coarse then the topology generated by the product criteria. This has an
important operational consequence: to show that a bivariate function $f(x,y)$
is continuous in the product topology, it is generally insufficient to show
that $f(x_\lambda,y_\lambda)\to f(x,y)$ whenever $x_\lambda\rightarrow y$
(generating) and $y_\lambda\rightarrow y$ (generating).

\begin{definition}\label{df:topological-convergence-criteria}
A convergence criteria $\gamma$ is \emph{topological} if every convergent net in
$\gamma^\dagger$ satisfies $\gamma$.
\end{definition}

Some primitive notations are useful. Logical operations will be extended to the
convergence criteria with the obvious meaning: for example,
$\gamma_1\vee\gamma_2(x_\lambda\rightarrow
x)\equiv\gamma_1(x_\lambda\rightarrow x)\vee\gamma_2(x_\lambda\rightarrow x)$,
i.e., the logical ``and'' of $\gamma_1$ and $\gamma_2$.

\begin{theorem}\label{th:convergence-topology-lattice}
Suppose $\gamma_1$ and $\gamma_2$ are convergence criteria.
\begin{enumerate}\alphalist
  \item
    If $\gamma_1\Rightarrow\gamma_2$ then
    $\gamma_1^\dagger\supseteq\gamma_2^\dagger$ (relaxed convergence criteria
    generate finer topologies). The convergence criteria defines as all nets
    converge [only the constant nets converge] generates the discrete
    [indiscrete] topology.
  \item
    $(\gamma_1\vee\gamma_2)^\dagger=\gamma_1^\dagger\wedge\gamma_2^\dagger$,
    (the logical ``or'' of two criteria generates the intersection of the
    topologies generated by the criteria separately).
  \item
    $(\gamma_1\wedge\gamma_2)^\dagger\supseteq\gamma_1^\dagger\vee\gamma_2^\dagger$,
    (the logical ``and'' of two criteria generates a topology finer than the
    join of the topologies generated by the criteria separately). If $\gamma_1$
    and $\gamma_2$ are both topological then
    $(\gamma_1\wedge\gamma_2)^\dagger=\gamma_1^\dagger\vee\gamma_2^\dagger$.
\end{enumerate}
\end{theorem}

\begin{proof}
(a)
Set $\tau_i\equiv\gamma_i^\dagger$, $i=1,2$.  Suppose $E$ is closed in
$\tau_2$. If $\gamma_1(x_\lambda\rightarrow x)$ is true with $x_\lambda\in E$
then $\gamma_2(x_\lambda\rightarrow x)$ is true, from which $x\in E$. Thus $E$
is closed in $\tau_1$, so $\tau_2\subseteq\tau_1$, i.e.,
$\gamma_1^\dagger\supseteq\gamma_2^\dagger$. The last two statements follow
from this, or directly: If every net converges then any constant net of any
point in any set converges to any point not in that set, from which the only
closed sets are the empty set and the whole space, and the topology generated
is indiscrete. If no net converges then the condition that a set be closed is
vacuously true for any set, and the topology generated is indiscrete.

(b)
Suppose $E$ is closed in $\tau_1\wedge\tau_2$. Since
$\tau_1\wedge\tau_2\subseteq\tau_1$ and $\tau_1\wedge\tau_2\subseteq\tau_2$,
$E$ is closed in both $\tau_1$ and $\tau_2$. So if
$\gamma_1\vee\gamma_2(x_\lambda\rightarrow x)$ is true then one of
$\gamma_1(x_\lambda\rightarrow x)$ or $\gamma_2(x_\lambda\rightarrow x)$ is,
and $x\in E$ in either case. This shows $E$ is closed in
$(\gamma_1\vee\gamma_2)^\dagger$ and hence that
$\tau_1\wedge\tau_2\subseteq(\gamma_1\vee\gamma_1)^\dagger$. Conversely,
$\gamma_1\Rightarrow\gamma_1\vee\gamma_2$ and
$\gamma_2\Rightarrow\gamma_1\vee\gamma_2$, from which
$\tau_1\supseteq(\gamma_1\vee\gamma_2)^\dagger$ and
$\tau_2\supseteq(\gamma_1\vee\gamma_2)^\dagger$ and
$(\gamma_1\vee\gamma_2)^\dagger\subseteq\tau_1\cap\tau_2=\tau_1\wedge\tau_2$.

(c)
If $\gamma_1\wedge\gamma_2\Rightarrow\gamma_1$ and
$\gamma_1\wedge\gamma_2\Rightarrow\gamma_2$ then
$(\gamma_1\wedge\gamma_2)^\dagger\supseteq\tau_1$ and
$(\gamma_1\wedge\gamma_2)^\dagger\supseteq\tau_2$. But $\tau_1\vee\tau_2$ is
the coarsest topology containing both $\tau_1$ and $\tau_2$, so
$\tau_1\vee\tau_2\subseteq(\gamma_1\wedge\gamma_2)^\dagger$.  For the second
part, let $\hat\gamma,\hat\gamma_1,\hat\gamma_2$ be the convergence criteria
defined by the topologies $\tau_1\vee\tau_2$, $\tau_1$, and $\tau_2$,
respectively. The join topology $\tau_1\vee\tau_2$ contains both $\tau_1$ and
$\tau_2$, so $\hat\gamma\Rightarrow\hat\gamma_1$ and
$\hat\gamma\Rightarrow\hat\gamma_2$.  Since $\hat\gamma_1=\gamma_1$ and
$\hat\gamma_2=\gamma_2$, it follows that
$\hat\gamma\Rightarrow\gamma_1\wedge\gamma_2$, so
$\gamma_1^\dagger\vee\gamma_2^\dagger=\tau_1\vee\tau_2=\hat\gamma^\dagger\supseteq(\gamma_1\wedge\gamma_2)^\dagger$.
\end{proof}

With respect to~(c),
$(\gamma_1\wedge\gamma_2)^\dagger\supsetneq\gamma_1^\dagger\vee\gamma_2^\dagger$
is possible: consider the three point set $\{1,2,3\}$ with the topology
$\tau_1=\{\emptyset,1,12,123\}$. This is one of the known 29~topologies on
three point sets. Let $\gamma_1$ be the convergence criteria defined as every
constant net converges to its value, the constant net $1$ also converges to
$2$, and the constant net $2$ also converges to $3$.  As is easily verified
(suppress the set braces)
\begin{equation*}\begin{split}
  &\onm{pcl}_{\gamma_1}(\emptyset)=\emptyset,\quad
  \onm{pcl}_{\gamma_1}(1)=12,\quad
  \onm{pcl}_{\gamma_1}(2)=23,\quad
  \onm{pcl}_{\gamma_1}(3)=3,\\
  &\onm{pcl}_{\gamma_1}(23)=23,\quad
  \onm{pcl}_{\gamma_1}(12)=\onm{pcl}_{\gamma_1}(13)=\onm{pcl}_{\gamma_1}(123)=123,\quad
\end{split}\end{equation*}
so that the closed sets are $\emptyset, 3, 23, 123$, and the open sets, being
complements of closed sets, are $123,12,1,\emptyset$, i.e., $\gamma_1$
generates $\tau_1$. In the topology $\tau_1$, the only open set containing $3$
is $123$, so every net converges to $3$, and in particular, the constant net
$1$ converges to $3$. So $\gamma_1$ is not topological since that net does not
converge by $\gamma_1$. Similarly, the convergence criteria $\gamma_2$ defined
as every constant net converging to its value, the constant net $1$ converges
to $3$, and the constant net $3$ converges to $2$ generates the topology
$\tau_2=123,13,1,\emptyset$. The join topology $\tau_1\vee\tau_2$ is
$\tau_1\cup\tau_2=\emptyset,1,12,13,123$ but in $\gamma_1\wedge\gamma_2$ only
the constant nets converge to themselves and $(\gamma_1\wedge\gamma_2)^\dagger$
is the discrete topology.

Recall the familiar inductive limit topology on subsets: suppose that $I$ is a
directed set, $X_i\subset X$, and $X_i$ are topological spaces such that
$X_i\subseteq X_j$ with continuous inclusion whenever $i\le
j$. The \emph{inductive limit topology} is the finest topology such that
every inclusion $X_i\rightarrow X$ is continuous. A subset $U$ is open in the
inductive limit topology if and only if $U\cap X_i$ is open for all $i$.
Consider the convergence criteria: $x_\alpha\rightarrow x$ if there are an
$\alpha^*$ and $i$ such that $x_\alpha\in X_i$ whenever $\alpha\ge\alpha^*$,
and $x_\alpha\rightarrow x$ in the topology of $x_i$. Call this criteria
\emph{eventual membership convergence}. The inductive limit topology is
especially convenient with respect to convergence:

\begin{theorem}\label{th:membership-convergence-generates-inductive-limit}
Eventual membership convergence generates the inductive limit topology.
\end{theorem}

\begin{proof}
There is a general approach for such results: to show a convergence topology is
equal to another given topology, show first that the convergence criteria
converges in the topology, so that
Theorem~\ref{th:convergence-topology-lattice}(a) implies the convergence
topology is finer. In the case that the given topology is the finest satisfying
some condition, then showing that the convergence topology satisfies that same
condition completes the proof. So, suppose that $x_\alpha\rightarrow x$
(eventual membership), and choose $\alpha^*$ and $i$ as in the definition. Then
the inclusion of $X_i\rightarrow X$ is continuous in the inductive limit
topology, so $x_\alpha$ converges in in that. Conversely, pick any $i$, suppose
$x_\lambda\in X_i$ and $x_\lambda\rightarrow x$ in the topology of $X_i$, and
let $\iota\colon X_i\rightarrow X$ be the inclusion. Then
$\iota(x_\lambda)=x_\lambda$ and $x_\lambda$ satisfies the eventual membership
convergence criteria, so $\iota(x_\lambda)$ converges in the topology generated
by that. Thus $\iota$ is continuous
by Theorem~\ref{th:preconvergence-suffices-closed-continuous}, but the inductive limit
topology is the finest in which each such inclusion is continuous.
\end{proof}

Suppose $X$ is a topological space, $\pi\colon X\to Y$ is onto, and define the
convergence criteria $\gamma$ by $y_\lambda\to y$ if there is a net $x_\lambda$
such that $\pi(x_\lambda)=y_\lambda$, $x_\lambda$ converges to some $x$, and
$\pi(x)=y$. Let $\tau$ be the quotient topology on $y$. Then $\pi$ is
continuous in the quotient topology on $Y$, so $\gamma$-convergence implies
$\tau$-convergence and $\gamma^\dagger\supseteq\tau$. The quotient topology
$\tau$ is the finest topology such that $\gamma$ is continuous, and, if
$x_\lambda\to x$ in $X$ then $\gamma$ is true for $\pi(x_\lambda)\to\pi(x)$ and
so $\pi$ is continuous in $\gamma^\dagger$. Hence $\gamma$ generates the
quotient topology. Let $X=\set{(x,1/x)}{x>0}\cup\sset{0}\times\bbR$, i.e., the
union of the graph of $y=1/x$ and the $y$-axis, with the subspace topology from
$\bbR^2$. Define $Y=[0,\infty)$, with the usual topology, and 
$\pi\colon X\to Y$ by $\pi(x,y)=x$. Then $\pi$ is a quotient map and
$1/n\rightarrow 0$ in $Y$ but there is no convergent $(x_n,y_n)\in X$ and a
$y$ such that $(x_n,y_n)\to (0,y)$. Thus there are convergent nets in $X$
that are not convergent in the $\gamma$ topology.

\end{document}